\documentclass{article}
\usepackage{typearea}
\usepackage{amsmath,amssymb}
\usepackage{bm}
\usepackage[dvipdfmx]{graphicx}
\usepackage{float}
\usepackage[hang,small,bf]{caption}
\usepackage[subrefformat=parens]{subcaption}
\usepackage{amsthm}
\newtheorem{defi}{Definition}[section]
\newtheorem{theo}[defi]{Theorem}
\newtheorem{lemm}[defi]{Lemma}

\newtheorem{prop}[defi]{Proposition}
\newtheorem{co}[defi]{Corollary}

\newtheorem{rem}[defi]{Remark}

\usepackage{usebib}
\bibinput{kajihara_reference}
\usepackage{makeidx}

\usepackage{ascmac} 
\usepackage{fancybox}
\usepackage{float}
\usepackage{tabularx}
\usepackage{mathtools}
\usepackage{here}

\newcommand{\cals}{\mathcal{S}}

\newcommand{\R}{\mathbb{R}}
\newcommand{\Z}{\mathbb{Z}}
\newcommand{\T}{\mathbb{T}}
\newcommand{\N}{\mathbb{N}}
\newcommand{\Q}{\mathbb{Q}}

\newcommand{\per}{\mathrm{per}}

\newcommand{\calm}{\mathcal{M}}

\newcommand{\xkr}{X_{k,\rho}}

\newcommand{\tilj}{\tilde{J}}

\title{Variational Structures for Infinite Transition Orbits \\of Monotone Twist Maps}
\author{Yuika Kajihara}
\date{\today}

\begin{document}
\maketitle

\begin{abstract}
In this paper, we consider chaotic dynamics and variational structures of area-preserving maps.
There is a lot of study on dynamics of their maps and 
the works of Poincare and Birkhoff are well-known.
To consider variational structures of area-preserving maps,  we define a special class of area-preserving maps called {\it monotone twist maps}.
Variational structures determined from twist maps can be used for constructing characteristic trajectories of twist maps.
Our goal is to define the variational structure such as giving {\it infinite transition orbits} through minimizing methods.
\end{abstract}

\section{Introduction}
\label{section:intro}
%本結果の背景
%Rabinowitzの結果の説明
In this paper, we consider  chaotic dynamics and variational structures of area-preserving
%(we write AP in short)
maps.
The dynamics of such maps have been widely studied, with key findings by  Poincar\'e and Birkhoff.
To explore these variational structures, 
we define a special class of  area-preserving maps called {\it monotone twist maps}:

\begin{defi}[monotone twist maps]
Set a map $f \colon \T \times \R \to \T \times \R$ and assume that $f \in C^1$ and
a lift $\tilde{f}$ of $f$ $\colon \R \times \R \to  \R \times \R$, $(x,y) \mapsto (X,Y)$ satisfy the followings:
\begin{itemize}
  \item[($f_1$)] $\tilde{f}$ is  area-preserving, i.e., $dx \wedge dy = dX \wedge dY$, and
  \item[($f_2$)] $\partial X / \partial y>0$ (twist condition).
\end{itemize}
Then $f$ is said to be a monotone twist map.
\end{defi}
By Poincar\'e's lemma, we get a generating function $h$ for a monotone twist map $f$
and  it satisfies
$
dh=YdX -ydx.
$
That is,
\[
y= \partial_1 h(x,X), \
Y= -\partial_2 h(x,X).
\]
For the above $h$, by abuse of notation, we define $h \colon \R^{n+1} \to \R$ by:
\begin{align}
\label{action_n}
h(x_0,x_1, \cdots, x_n)=\sum_{i=1}^{n} h(x_i,x_{i+1}).
\end{align}
We can regard $h$ as a variational structure associated with $f$, because any critical point of $\eqref{action_n}$, say $(x_0,\cdots,x_n)$, gives us an orbit of $\tilde{f}$ by $y_i=-\partial h_1(x_i,x_{i+1})=\partial h_2(x_{i-1},x_{i})$.
This is known as the Aubry-Mather theory, which is so called because
Aubry studied critical points of the action $h$ in \cite{AubrDaer1983} and Mather developed the idea (e.g. \cite{Mather1982, Mather1987}).
%Bangert \cite{Bangert1988} investigates good conditions of $h$ for study in  {\it{minimal sets}}.
%Here we will give a summary of  \cite{Bangert1988}.
We briefly summarize Bangert's investigation of good conditions of $h$ for study in minimal sets \cite{Bangert1988}.

We consider the space of bi-infinite sequences of real numbers and define convergence of a sequence $x^n \in \R^\Z$  to $x \in \R^\Z$ by:
\begin{align}
\label{t_metric}
\lim_{n \to \infty} |x^n_i - x_i| = 0 \ ({}^{\forall} i \in \Z).
\end{align}
Assume that a Lipschitz continuous map $h \colon \R^2  \to \R$ satisfies $(h_1)$-$(h_4)$, where
they are given by the followings:
\begin{itemize}
  \item[($h_1$)] For all $(\xi,\eta) \in \R^2$, $h(\xi,\eta)=h(\xi+1,\eta+1)$;
  \item[($h_2$)] $\displaystyle{\lim_{\eta \to \infty} h(\xi,\xi+\eta) \rightrightarrows \infty}$;
  \item[($h_3$)] If $\underline{\xi}< \bar{\xi}$ and $\underline{\eta}< \bar{\eta}$, then
 $ h(\underline{\xi},\underline{\eta}) + h(\bar{\xi},\bar{\eta}) <   h(\underline{\xi},\bar{\eta}) + h(\bar{\xi},\underline{\eta}) $; and
  \item[($h_4$)] If $(x,x_0,x_1)$ and $(\xi,x_0,\xi_1)$ are minimal and $(x,x_0,x_1) \neq (\xi,x_0,\xi_1)$, then
  $(x-\xi)(x_1-\xi_1)<0$.
\end{itemize}

\begin{defi}[minimal configuration/stationary configuration]
A finite configuration $x=(x_i)_{n \le i \le m}$ is said to be minimal if,
for any finite configuration $\{y_i\}_{i=n_0}^{n_1}$ with $y_{n_0}=x_{n_0}$ and $y_{n_1}=x_{n_1}$,
\[
h(x_{n_0},x_{n_0+1}, \cdots, x_{n_1-1}, x_{n_1}) \le h(y_{n_0},y_{n_0+1}, \cdots, y_{n_1-1}, y_{n_1}),
\]
where  $n \le n_0 < n_1 \le m$.
An infinite configuration $x=(x_i)_{i \in \Z}$ is called minimal if, for any $n<m$, we have $x=(x_i)_{i=n}^{m}$ is minimal.
Moreover, a configuration $x$ is called locally minimal or a stationary configuration if for any $i \in \Z$,
it holds that
$
\partial_2 h (x_{i-1},x_i) + \partial_1h(x_i,x_{i+1})=0.
$
\end{defi}

For $x=(x_i) \in \R^\Z$, we define:
\[
\alpha^+(x) :=\lim_{i \to \infty} \frac{x_i}{i},\
\alpha^-(x) :=\lim_{i \to -\infty} \frac{x_i}{i}.
\]
\begin{defi}[rotation number]
If both $\alpha^+(x)$ and $\alpha^-$(x) exist and $\alpha^+(x)=\alpha^-(x)(=:\alpha(x))$, 
then we call $\alpha(x)$ a rotation number of $x$.
\end{defi}

Let $\calm_\alpha$ be a minimal set consisting of minimal configurations with rotation number $\alpha$.
It is known that for any $\alpha \in \R$, the set $\calm_{\alpha}$ is non-empty and compact (see  \cite{Bangert1988} for the proof).

\begin{defi}[periodic orbits]
\label{def:periodic}
A configuration $ x=(x_i)$ is said to be  $(q,p)$-periodic
if $x=(x_i) \in \R^\Z$ satisfies
\[
x_{i+q}=x_{i} + p,
\]
for any $i \in \Z$.
\end{defi}
It is easily seen that if $x$ is $(q,p)$-periodic, then its rotation number is $p/q$.
This chapter discusses the case where $\alpha \in \Q$.
For $\alpha=p/q \in \Q$, we set:
\[
\calm_{\alpha}^\per := \{ x \in \calm_{\alpha} \mid  \text{$x$ is $(q,p)$-periodic}\} \cap \calm_{\alpha}.
\]

\begin{defi}[neighboring]
A pair of $(q,p)$-periodic minimal configurations $x^0,x^1$ with $x^0 \neq x^1$
is called neighboring if there is no other $x \in \calm_{\alpha}^\per$ with $x^0 < x<x^1$
\end{defi}

Given a neighboring pair $(x^0,x^1)$, we define:
\begin{align*}
\calm_{\alpha}^+(x^0,x^1)&=\{x \in \calm_{\alpha} \mid |x_i - x^0| \to 0 \ (i \to -\infty) \ and \ |x_i - x^1| \to 0 \ (i \to \infty) \} \ and \\
\calm_{\alpha}^-(x^0,x^1)&=\{x \in \calm_{\alpha} \mid |x_i - x^0| \to 0 \ (i \to \infty) \ and \ |x_i - x^1| \to 0 \ (i \to -\infty) \}.
\end{align*}
Bangert \cite{Bangert1988} showed the following proposition.
\begin{prop}
Given $\alpha \in \Q$, $\calm_{\alpha}^\per$ is nonempty.
Moreover, 
if $\calm_\alpha$ has a neighboring pair $(x^0,x^1)$, 
then $\calm_{\alpha}^+$ and $ \calm_{\alpha}^-$ are nonempty.
\end{prop}

Although we have discussed minimal configuration in the preceding paragraph, 
there are also interesting works that treat non-minimal orbits between periodic orbits,
particularly,  \cite{Rabinowitz2008} and \cite{Yu2022}.
In  \cite{Rabinowitz2008}, Rabinowitz used minimizing methods to prove the existence of three types of solutions-periodic, heteroclinic and homoclinic-in potential systems with reversibility for time, i.e. $V(t,x)=V(-t,x)$.
Under an assumption called  {\it gaps}, which is similar to {neighboring}, for periodic and heteroclinic solutions, 
each non-minimal heteroclinic and homoclinic orbit can be given as {\it{n transition orbits}} ($n \ge 2$)
between two periodic orbits.

\begin{defi}[$n$ transition orbits]
\label{defi:transition}
An orbit or solution is called an $n$ transition orbit
if
it passes between two periodic orbits, say $u^0$ and $u^1$, and alternately through a neighborhood of them,
where `$n$' means the number of times it changes from a neighborhood of $u^0$ to a neighborhood of $u^1$.
%Thus odd transitions imply heteroclinic orbits and even do homoclinic orbits.
\end{defi}
In Definition \ref{defi:transition}, these are heteroclinic orbits when $n$ is odd and these are homoclinic orbits when $n$ is even.

\if0
\begin{figure}[H]
  \begin{minipage}[b]{0.48\linewidth}
    \centering
    \includegraphics[width=6cm]{1transition.pdf}
    \subcaption{$1$ transition heteroclinic configuration}
  \end{minipage}
  \begin{minipage}[b]{0.48\linewidth}
    \centering
    \includegraphics[width=6cm]{2transition.pdf}
    \subcaption{$2$ transition homoclinic configuration}
  \end{minipage}
  \caption{One and two transition orbits}
\end{figure}
\fi

\begin{rem}
The existence of $1$-transition orbits ( i.e., minimal heteroclinic orbits)  does not require gaps for heteroclinic orbits as seen in Section \ref{sec:general}.
Also, we can see this by considering a simple pendulum system.
%(b) Definition \ref{defi:transition} and `$k$-transition' in \cite{RabiStre2011} are different. 
%The latter means a solution that passes between $k-1$ periodic orbits in PDEs.
\end{rem}

Rabinowitz's approach can be applied to variational methods for area-preserving maps. 
Yu \cite{Yu2022} added $h$ to the following assumption $(h_5)-(h_6)$ to $h$:
\begin{itemize}
  \item[($h_5$)] There exists a positive continuous function $p$ on $\R^2$ such that:
  \[
  h(\xi,\eta') +  h(\eta,\xi') -   h(\xi,\xi') -  h(\eta,\eta')  > \int_{\xi}^{\eta}  \int_{\xi'}^{\eta'} p
  \]
  if $\xi < \eta$ and $\xi' < \eta'$.
  \item[($h_6$)] There is a $\theta > 0$ satisfying the following conditions:
  \begin{itemize}
  \item $\xi \mapsto \theta \xi^2 /2 - h(\xi,\xi')$ is convex for any $\xi'$, and
  \item $\xi \mapsto \theta \xi'^2 /2 - h(\xi,\xi')$ is convex for any $\xi$.
  \end{itemize}
\end{itemize}

\begin{rem}
\label{rem:delta_bangert}
One of a sufficient conditions for $(h_2)-(h_5)$ is 
\[
(\tilde{h}) \ \text{$h \in C^2$ and $\partial_1 \partial_2 h \le -\delta <0$ for some $\delta>0$}.
\]
by taking a constant function $\rho=\delta$ as a positive function.
If $f \in C^1$ and satisfies $\partial X/\partial y \ge \delta$ for some $\delta>0$, a generating function $h$ for $f$ satisfies $(\tilde{h})$.
However, $(\tilde{h})$ is not a necessary condition for satisfying $(h_2)-(h_5)$.
%as seen in Section \ref{sec:example_h}.
\end{rem}
 
 Clearly, $(h_5)$ implies $(h_3)$.
Mather \cite{Mather1987} proved
that  if $h$ satisfies $(h_1)$-$(h_6)$, then $\partial_2 h (x_{i-1},x_i)$ and
 $\partial_1h(x_i,x_{i+1})$ exist in the meaning of the left-sided limit.
 In addition, he proved that if $x$ is a locally minimal configuration, then it satisfies:
\[
\partial_2 h (x_{i-1},x_i) + \partial_1h(x_i,x_{i+1})=0.
\]

Yu applied Rabinowitz's methods to twist maps to show finite transition orbits of monotone twist maps for all $\alpha \in \Q$.
We will give a summary of his idea in the case of $\alpha=0$ (i.e. $(q,p)=(1,0)$ in Definition \ref{def:periodic}).
 Let $(u^0,u^1)$ be a neighboring pair with $\alpha=0$.
 By abuse of notation, we then denote $u^j$  by the constant configuration $\{x_i=u^j\}_{i\in\Z}$.
We set:
\[
{c:=\min_{ x \in \R} h(x,x)}(=h(u^0,u^0)=h(u^1,u^1)).
\]
And:
\begin{align}
\label{pre-action}
I(x) := \sum_{i \in \Z } a_i(x),
\end{align}
where $a_i(x)= h(x_i,x_{i+1}) - c$.

\if0
\begin{rem}
In the case of the reversible potential systems in \cite{rabi}, one shows $a_i(x)$ is non-negative, i.e., $I$ is a positive term series.
Though we possibly take $h(x,y)<0$,
we show an example such that $h(x,y)$ is non-negative.
\end{rem}
\fi

Yu studied local minimizers of $I$ to show the existence of finite transition orbits.
Given a rational number $\alpha \in \Q$ and a neighboring pair  $(x^0,x^1)$ with $x^0, x^1 \in \calm_{\alpha}^\per$,
we let:
\begin{align*}
I^+_{\alpha}(x^0,x^1)&=\{x_0 \in \R \mid x \in \calm_{\alpha}^+(u^0,u^1)\}, \ and\\
I^-_{\alpha}(x^0,x^1)&=\{x_0 \in \R \mid x \in  \calm_{\alpha}^-(u^0,u^1) \}.
\end{align*}
%where $\calm_{\alpha}^0(u^0,u^1)$ is a subset of $\calm_{\alpha}^+(u^0,u^1)$ and
%$\calm_{\alpha}^1(u^0,u^1)$ is also a subset of $\calm_{\alpha}^-(u^0,u^1)$.
%We give each precise definition of $\calm_{\alpha}^i(u^0,u^1)$ $(i=0,1)$ in Section \ref{sec:pre}.
Under the above setting, he showed the following theorem.
\begin{theo}[Theorem 1.7, \cite{Yu2022}]
\label{theo:yu_main}
Given a rational number $\alpha \in \Q$ and a neighboring pair $(x^0,x^1)$  with $x^0, x^1 \in \calm_{\alpha}^\per$.
If
\begin{align}
\label{hetero_gap}
I^+_{\alpha}(x^0,x^1) \neq (x_0^0,x_0^1) \ \text{and} \
I^-_{\alpha}(x^0,x^1) \neq (x_0^0,x_0^1) ,
\end{align}
then, for any $k \in \Z_{>0}$, there exist infinite number of $2k$-transition heteroclinic orbits and $2k+1$-transition homoclinic orbits  which passes through the neighborhood of  $x^0$ and $x^1$ alternately.
\end{theo}
Furthermore, Rabinowitz \cite{Rabinowitz2008} proved the existence of an infinite transition orbit as a limit of sequences of finite transition orbits.
However, the variational structure of infinite transition orbits is an open question.
To consider the question for twist maps, the following proposition is essential.
\begin{prop}[Proposition 2.2, \cite{Yu2022}]
\label{prop:finite}
If $I(x) < \infty$, then $|x_i -u^1| \to 0$ or $|x_i -u^0| \to 0$ as $|i| \to \infty$.
\end{prop}
Since this implies that $I(x) = \infty$  if $x$ is an infinite transition orbit, 
we need to fix normalization of $I$.
Therefore, we focus on giving the variational structure of infinite transition orbits of twist maps
and discuss the conditions of $h$.
Roughly speaking, our main theorem and the steps of the proof imply the following theorem.
\begin{theo}
The function $J$ defined in Section \ref{subsec:setting}  gives us the variational structure of infinite transition orbits.
\end{theo}

This paper is organized as follows.
Section \ref{sec:pre} deals with some results in \cite{Yu2022} and remarks.
In Section \ref{sec:pf}, our main results are stated and proved in the case of $\alpha=0$.
Section \ref{sec:example_h} provides the special case of $h$.
In the last section, we extend the result of the infinite transition orbits to a more general representation.
%In the fourth section, we give interesting examples.

%\newpage

\section{Preliminary}
\label{sec:pre}
In this section, we would like to introduce properties of $\eqref{pre-action}$ and minimal configurations using several useful results in \cite{Yu2022}.
%Though we do not use the assumption $(h_1)-(h_6)$ directly for our proof in Section \ref{sec:pf},they are needed because we use lemmata and propositions in \cite{Yu2022}.
Moreover, we study estimates of heteroclinic configurations.

\subsection{Properties of minimal configurations}
Let $(u^0,u^1)$ be a neighboring pair of $\calm_0^\per$ and:
\begin{align}
\label{eq:x}
\begin{split}
X &= X(u^0,u^1)= \{ x = (x_i)_{i \in \Z} \mid u^0 \le x_i \le u^1 \ ({}^{\forall} i \in \Z) \},\\
X(n)&=X(n ; u^0,u^1)=\{ x =(x_i)_{i=0}^{n}  \mid u^0 \le x_i \le u^1 \ ({}^{\forall} i \in \{0, \cdots, n\}) \}, and\\
\hat{X}(n) &= \hat{X}(n ; u^0,u^1)=\{ x =(x_i)_{i=0}^{n}   \mid x_0=x_n, \ u^0 \le x_i \le u^1 \ ({}^{\forall} i \in \{0, \cdots, n\}) \}.
\end{split}
\end{align}

\begin{defi}[\cite{Yu2022}]
For $x \in X$, we set:
\[d(x):=\max_{0 \le i \le n} \min_{j \in \{0,1\}} |x_i -u^j|.\]
For any $\delta > 0$, let
$C$ be a Lipschitz constant of $h$ and:
\[
\phi(\delta):= \inf_{n \in \Z^+}
\inf \left\{
\sum_{i=0}^{n-1} a_i(x) \mid x \in \hat{X}(n) \ {\text{and}} \ d(x) \ge \delta
\right\}.
\]
\end{defi}

\begin{rem}
(1)We can replace Lipschitz continuity with local Lipschitz continuity on $[u^0,u^1] \times [u^0,u^1] $.
(2)The function $\phi(\delta)$ is positive for any $\delta>0$ and $\phi(0)=0$.
\end{rem}

\begin{lemm}[Lemma 2.7 and 2.8, \cite{Yu2022}]
\label{lemm:finite_bdd}
For any $n \in \N$ and $x \in \hat{X}(n)$ satisfying $\displaystyle{\min_{0 \le i \le n}  |x_i-u^j| \ge \delta}$ for $j=0,1$, 
\label{lemm:yu_lower}
\[\sum_{i=0}^{n-1} a_i(x) \ge n \phi(\delta)\]
and for any $n \in \N$ and $x \in {X}(n)$,
\[
\sum_{i=0}^{n-1} a_i(x)  \ge -C|x_n-x_0|.\]
\end{lemm}
\begin{proof}
See \cite{Yu2022}. This proof requires $(h_3)$.
\end{proof}

\begin{lemm}[Lemma 2.9, \cite{Yu2022}]
\label{lemm:positive}
If $x \in X$ satisfies  $|x_i - u^0|$ as $|i| \to \infty$ and $x_i \neq u^0$ for some $i \in Z$,
then $I(x)>0$.
\end{lemm}

We also need to check that each component of stationary configurations is not equal to $u^0 $ or $u^1$.
This follows from the next lemmas.
\begin{lemm}[Lemma 2.11, \cite{Yu2022}]
\label{lemm:estimate_periodic1}
For any $\delta \in (0, u^1-u^0]$, if $(x_i)_{i=0}^{2}$ satisfies
\begin{enumerate}
\item $x_i \in [u^0,u^1]$ for all $i=0,1,2$;
\item $x_1 \in [u^1-\delta,u^1]$, and $x_0 \neq u^1$ or $x_2 \neq u^1$; and
\item $h(x_0,x_1,x_2) \le h(x_0, \xi,x_2)$ for all $\xi \in [u^1-\delta,u^1] $,
\end{enumerate}
then $x_1 \neq u^1$.
This still holds if we replace every $u^1$ by $u^0$ and every $[u^1-\delta,u^1]$ by $[u^0, u^0+\delta]$.
\end{lemm}

\begin{lemm}[Lemma 2.12, \cite{Yu2022}]
\label{lemm:estimate_periodic2}
If a finite configuration $x=(x_i)_{i=n_0}^{n_1}$ satisfies
\begin{enumerate}
\item $x_i \in [u^0,u^1]$ for all $i=n_0, \cdots, n_1$ and
\item for any $(y_{i})_{i=n_0}^{n_1}$ satisfying $y_{n_0}=x_{n_0}$, $y_{n_1}=x_{n_1}$, and $y_i \in [u^0,u^1]$,
\[
h(x_{n_1},x_{n_1+1}, \cdots, x_{n_2-1}, x_{n_2}) \le h(y_{n_1},y_{n_1+1}, \cdots, y_{n_2-1}, y_{n_2}),
\]
\end{enumerate}
then $x$ is a minimal configuration.
Moreover, if $x$ also satisfies $x_{n_0} \notin \{u^0,u^1\}$ or  $x_{n_1} \notin \{u^0,u^1\}$,
then $x_i \notin \{u^0,u^1\} $ for all $i=n_{0}+1, \cdots, n_1-1$.
\end{lemm}
\begin{proof}[Proof of the two lemmas above]
See \cite{Yu2022}. These proofs require $(h_4)$ and $(h_5)$.
\end{proof}

Moreover, we can replace $\alpha=0$ with other rational numbers as seen below.
\begin{defi}[Definition 5.1, \cite{Yu2022}]
For $\alpha=p/q \in \Q \backslash \{0\}$, we set:
\begin{align*}
%X_{\alpha}(q;x^-,x^+)&:=\{x=(x_i)_{i=0}^{q} \mid x_i^- \le x_i \le x_i^+ (i=0, \cdots ,q)\}\\
%X_{\alpha}(q;x^-,x^+)&:=\{x \in X(q) \mid x_q =x_0 + p\}\\
X_{\alpha}(x^-,x^+)&:=\{x=(x_i)_{i \in Z} \mid x_i^- \le x_i \le x_i^+ (i \in Z)\}.
\end{align*}
where $x^-$ and $x^+$ are $(q,p)$-periodic neighboring minimal configurations with $x^- < x^+$.
\end{defi}

\begin{defi}[Definition 5.2, \cite{Yu2022}]
Let $h_i \colon \R^2 \to \R$ be a continuous function for $i=1,2$.
For $h_1$ and $h_2$, we define $h_1 \ast h_2 \colon \R^2 \to \R$ by
\[
h_1 \ast h_2(x_1,x_2) = \min_{\xi \in \R} (h_1(x_1, \xi) + h_2(\xi, x_2)).
\]
We call this the \it{conjunction} of $h_1$ and $h_2$.
\end{defi}

Using the conjunction, we denote $H \colon \R^2 \to \R$ for $\alpha=p/q$  by:
\[
H(\xi,\xi') = h^{*q}(\xi,\xi'+p),
\]
where
$
h^{*q}(x,y)=h_1 \ast h_2 \ast \cdots \ast h_q \  (h_i=h).
$

\begin{defi}[Definition 5.5, \cite{Yu2022}]
For any $y =(y_i) \in X(x_0^-,x_0^+)$,
we define $x=(x_i) \in X_{\alpha}(x^-,x^+)$ as follows:
\begin{enumerate}
\item $x_{iq}=y_i + ip$ and
\item $(x_j)_{j=iq}^{(i+1)q}$ satisfies
\[
h(x_{iq}, \cdots, x_{(i+1)q}) =H(x_{iq},x_{(i+1)q}) = H(y_i,y_{i+1}),
\]
i.e.,  $(x_j)_{j=iq}^{(i+1)q}$ is a minimal configuration of $h$.
\end{enumerate}
\end{defi}
Although we focus on the case of rotation number $\alpha=0$, we may apply our proof to all rational rotation numbers from the following.
\begin{prop}[Proposition 5.6, \cite{Yu2022}]
\label{prop:alpha_configuration}
Let $y \in X(x_0^-,x_0^+)$ and $x \in X_{\alpha}(x^-,x^+)$ be defined as above.
If $y$ is a stationary configuration of $H$, 
then $x$ must be a stationary configuration of $h$.
\end{prop}

\subsection{Asymptotic behavior on $I$}

Let $X^0$ and $X^1$ be given by:
\begin{align*}
X^0&=\{x \in X \mid |x_i - u^0| \to 0 \ (i \to \infty), |x_i - u^0| \to 0 \ (i \to -\infty)\} \ and\\
X^1&=\{x \in X \mid |x_i - u^0| \to 0 \ (i \to -\infty), |x_i - u^0| \to 0 \ (i \to \infty)\}.
\end{align*}
By considering a local minimizer (precisely, a global minimizer in $X^0$ or $X^1$),
Yu \cite{Yu2022} proved the existence of heteroclinic configurations, which Bangert showed in \cite{Bangert1988},
as per the following proposition.
\begin{prop}[Theorem 3.4 and Proposition 3.5, \cite{Yu2022}]
There exists a stationary configuration $x $ in $X^0$ (resp. $X^1$) satisfying $I(x)=c_0$ (resp. $I(x)=c_1$),
where
\[c_0=\inf_{x \in X^0} I(x), \ c_1=\inf_{x \in X^1} I(x)\]
Moreover, $x$ is strictly monotone, i.e., $x_i<x_{i+1}$ (resp. $x_i>x_{i+1}$)  for all $i \in Z$.
\end{prop}
Let:
\begin{align*}
&\calm^0(u^0,u^1)=\{x \in X \mid c_0=\inf_{x \in X^0} I(x)\} \ and\\
&\calm^1(u^0,u^1)=\{x \in X \mid c_1=\inf_{x \in X^1} I(x)\}.
\end{align*}

Set $c_\ast:=I(x^0) + I(x^1)$, where $x^i \in \calm^i$.
From the above and Lemma \ref{lemm:positive}, we immediately obtain the following corollary.
\begin{co}
\label{co:positive_het}
$c_{\ast}>0$
\end{co}
\begin{proof}
Choose  $x^0 \in \calm$ and $x^1 \in \calm$ arbitrarily.
From monotonicity, $x^0$ and $x^1$ intersect exactly once.
We define $x^+$ and $x^-$ in $X$ by
$x^+_i := \max\{x^0_i,x^1_i\}$ and $x^-_i := \min\{x^0_i,x^1_i\}$.
By $(h_3)$ and Lemma \ref{lemm:positive},
\[
c_{\ast}=I(x^0) + I(x^-) \ge I(x^+) + I(x^-) >0.
\]
This completes the proof.
\end{proof}

Next, we consider a minimal configuration under fixed ends.
For $0<a,b< (u^1-u^0)/2$, let:
\begin{align*}
Y^{\ast,0}(n,a,b) &= X(n) \cap \{x_0=u^0+a\} \cap  \{x_n=u^0+b\} \ and\\
Y^{\ast,1}(n,a,b) &= X(n) \cap \{x_0=u^1-a\} \cap  \{x_n=u^1-b\}
\end{align*}
and  $y^0(n,a,b)=(y^0_i) \in Y^{\ast,0}(n,a,b) $ be a finite configuration satisfying:
\[\sum_{i=0}^{n-1} {a_i}(y^0(n,a,b)) = \min_{x \in Y^{\ast,0}(n,a,b)} \sum_{i=0}^{n-1} {a_i}(x).
\]

The assertion of the next lemma may appear confusing, but it is useful for our proof in Section \ref{sec:pf}.
\begin{lemm}
\label{lemm:one_delta}
For any $m_0, m_1 \in \Z_{\ge 0} (m_0<m_1)$ and $\rho_0, \rho_1$, and $\delta$ with $\rho_0 +  \rho_1 + 2\delta <c_{\ast}/2C$,
there exists $n_0$ such that 
for all $n \ge n_0$,
there exists $l \in [m_0, n-m_1] \cap \Z$ satisfying
$|y^0_l(n,\rho_0, \rho_1) - u^0| \le \delta$.
A similar argument holds if $u^0$ and $y^0$ are replaced by $u^1$ and $y^1$.
\end{lemm}
\begin{proof}
We first consider the case $m_0=m_1=0$.
Let $z=(z_i)_{i=0}^{n}$ be given by  $z_0=y^\ast_0$, $z_n=y^\ast_n$ and $z_{i} = u^0$ otherwise.
Clearly, for any $n \in \Z_{>0}$, 
\[\sum_{i=0}^{n} a_{i}(y^\ast_l(n,\rho_0, \rho_1)) \le
\sum_{i=0}^{n} a_{i}(z) \le C(\rho_1+\rho_2) < \frac{c_{\ast}}{2}.\]
%For $i$ even, $A_i = a_i$ and 
On the other hand, Lemma \ref{lemm:finite_bdd} and the definition of $y^0$ imply that if 
$x \in X(n) \cap \{\min |x_i - u^j| \ge \delta\}$, then
\begin{align*}
\sum_{i=0}^{n} a_i(x) \ge n \phi(\delta) - C|\rho_1-\rho_0|.
\end{align*}
and $\sum_{i=0}^{n} a_i(x)>{c_{\ast}}/{2}$  for sufficiently large $n$, which is a contradiction.
The above remark implies that for sufficiently large $n$, there exists $l \in [0, n] \cap \Z$ 
such that $|y_l(n,\rho_0, \rho_1) - u^0| \le \delta$ or $|y_l(n,\rho_0, \rho_1) - u^1| \le \delta$.
That is, 
either of the following two conditions holds:
\begin{enumerate}
\item[(a)] There exists $l \in [0, n] \cap \Z$ such that $|x_l(n,\rho_0, \rho_1) - u^0| \le \delta$ and
$|x_i(n,\rho_0, \rho_1) - u^1|>\delta $ for all $i \in [0,n] \cap \Z$, or
\item[(b)] There exists $l \in [0, n] \cap \Z$ such that $|x_l(n,\rho_0, \rho_1) - u^1| \le \delta$. 
\end{enumerate}
To prove our claim for $m_0=m_1=0$, it suffices to show that the case of $(b)$  does not occur for sufficiently large $n$.
If $(b)$ holds, then by Corollary \ref{co:positive_het}, 
\begin{align*}
\sum_{i} a_i(x) > c_{\ast} -C(\rho_1+\rho_2) -2C\delta> \frac{c_{\ast}}{2},
\end{align*}
which is a contradiction. 
\if0
Next we fix $m_0,m_1 \in \Z_{>0}$ arbitrarily.
It also suffices to show that the case of $(b)$  does not occur.
By Lipschitz continuity, 
\[
\sum a_i(y^0(n,a,b)) + \sum a_i(y^0(n,a,b)) \le C(y^0_{m_0}-u^0)+ C(y^0_{n-m_1}-u^0)
\]
and if $x \in X(n) \cap \{\min |x_i - u^j| \ge \delta\}$, then
\begin{align*}
\sum_{i=m_0}^{n-m_1-1} a_i(x) \ge (n-m_1-m_0) \phi(\delta) - C|\rho_1-\rho_0|,
\end{align*}
which is contradiction for any $n$ satisfying $|n-m_1-m_0| \gg 1 $.
\fi
\end{proof}

In fact, $y^i(n,a,b)$ can be asymptotic to $u^i$ any number of times,
as per the following lemma.
\begin{lemm}
\label{lemm:finite_delta}
For any $k$, $\rho_0, \rho_1$ and $\delta$ with $\rho_0 +  \rho_1 + 2\delta <c_{\ast}/2C$,
there exists $n_0$ such that 
for all $n \ge n_0$,
there exist $l_1, \cdots, l_k \in [0, n] \cap \Z$ satisfying
$|y^0_{l_i}(n,\rho_0, \rho_1) - u^0| \le \delta$ for all $i =1, \cdots, k$.
A similar argument holds if $u^0$ and $y^0$ are replaced by $u^1$ and $y^1$.
\end{lemm}
\begin{proof}
We only discuss the case where $k=2$.
If we assume that the lemma is false,
then we set $j_1=\lfloor n/2 \rfloor$.
For some $j_0 \in  [0, n] $, there exists $x=(x_i)_{i=j_0}^{j_0+j_1}$
satisfying $|x_l(n,\rho_0, \rho_1) - u^0| > \delta$ for all $i \in [j_0,\cdots,j_0+j_1] \cap \Z$.
On the other hand, $(h_1)$ and Lemma \ref{lemm:one_delta} imply that for sufficiently large $j_1$, it holds that $\sum a_i(x) > c_{\ast}/2$,
which is a contradiction. 
Other cases are shown in the same way.
\end{proof}

\begin{lemm}
\label{lemm:finite_het}
For any $\epsilon >0$, there exist $n_0 \in \Z_{\ge 0}$ and $x \in \calm^0(u^0,u^1)$ such that $\sum_{i=0}^{n-1} a_i(x) \in (c_0 -\epsilon, c_0 + \epsilon)$ for all $n \ge n_0$.
\end{lemm}
\begin{proof}
For sufficiently large  $n$,
there exists $y \in \calm^0$ such that:
\[
y_{0}-u^0 < \epsilon/2C, \text{and} \ u^1-y_n< \epsilon/2C.
\]
By Lemma \ref{lemm:yu_lower} and the monotonicity of $y$, 
\begin{align*}
\left| \sum_{i=0}^{n-1} a_i(y) -c_0 \right|
= \left| \sum_{i <0} a_i(y) +  \sum_{i >n} a_i(y) \right|
\le C( (y_0-u^0) + (u^1-y_n))
<  \epsilon
\end{align*}
as desired.
\end{proof}

Next, we assume a gap condition, i.e., 
\begin{align}
\label{gap}
\text{$(u^0,u^1) \backslash I_0$ and $(u^0,u^1) \backslash I_1$ are nonempty sets},
\end{align}
where $I_0=I_{0}^+(u^0,u^1)$ and $I_1=I_{0}^-(u^0,u^1)$.
We will check the properties of the heteroclinic configurations.
Under $\eqref{gap}$, the following lemma holds.
\begin{lemm}[Proposition 4.1, \cite{Yu2022}]
\label{lemm:gap_estimate}
For any $\epsilon>0$, there exist $\delta_i$ $(i=1,2,3,4)$ and positive constants  $e_0=e_0(\delta_1,\delta_2)$ and $e_1=e_1(\delta_3,\delta_4)$ satisfying:
\begin{align*}
\inf\{ I(x) \mid x \in X^0, x_0 =u_0+\delta_1  \text{ or }  x_0=u_1-\delta_2\} &\ge c_0+e_0 \ and\\
\inf\{ I(x) \mid x \in X^1, x_0 =u_1-\delta_3  \text{ or }  x_0=u_0+\delta_4\} &\ge c_1+e_1.
\end{align*}
\end{lemm}

%\newpage

\section{Statements and proofs of our main theorem}
\label{sec:pf}
\subsection{Variational settings and properties of action $J$}
\label{subsec:setting}
%Set  $\tilde{I}_0:=(u_0,u_1) \backslash I_0$ and  $\tilde{I}_1:=(u_0,u_1) \backslash I_1$.
%We assume a gap condition for periodic and heteroclinic configurations, i.e., $\tilde{I}_0$ and $\tilde{I}_1$ are nonempty sets.
Let $(u^0,u^1)$ be a neighboring pair with $u^0,u^1 \in \calm^\per_{0}$ and set:
\begin{align}
\label{def:kp}
\begin{split}
K&=\left\{ k=(k_i)_{i \in \Z} \subset \Z \mid k_0=0,  k_i < k_{i+1} \right\}, \ and\\
P&=\left\{ \rho=(\rho_i)_{i \in \Z} \subset \R_{>0} \mid 0 <\rho_i <(u^1-u^0)/2, \  \sum_{i \in \Z} \rho_i <\infty  \right\}.
\end{split}
\end{align}

For $k \in K$ and $\rho \in P$, the set $\xkr$ is given by:
\[
%\xkr =\{ x \in \R^\Z  \mid  \text{$0 \le  x_{k_i} - u_{k_i}^0 \le \rho_i$ if $i \equiv 0, 1$  and  $0 \le  u_{k_i}^1 - x_{k_i} \le \rho_i$ if $i \equiv -1, 2$} \}
\xkr =\left( \bigcap_{i \equiv 0,1}  Y^0(k_i,\rho_i) \right) \cap \left( \bigcap_{i \equiv -1,2}  Y^1(k_i,\rho_i) \right) 
\]
where
\[
Y^j(l,p) = \{x \in X \mid |x_l - u^j| \le p \} \ (j=0,1)
\]
 and $ a \equiv b$ means $a \equiv b \ (\mathrm{mod} \ 4)$.
 (See $\eqref{eq:x}$ for the definition of $X$.)
 
 \if0
\begin{figure}[H]
    \centering
    \includegraphics[width=8.5cm]{infinite_transition.png}
  \caption{An element of $\xkr$}
\end{figure}
\fi

To see the topological properties of $\xkr$, we first show the following lemma.
\begin{lemm}
\label{lemm:cpt}
The set $X$  is sequentially compact.
\end{lemm}
\begin{proof}
By Tychonoff's theorem, $X$ is a compact set.
It suffices to check that $X$ is metrizable.
Let $d \colon X \times X  \to \R$ be given by
\begin{align*}
d(x,y) &= \sum_{i \in \Z} \frac{|x_i - y_i|}{2^{|i|}}
\end{align*}
Clearly, $d$ is a metric function.
We show that convergence on  $d$ and  $\eqref{t_metric}$ is equivalent.
Since for all $i \in \Z$
\[
\frac{|x_i - y_i|}{2^{|i|}} \le  d(x,y),
\]
it is sufficient to show that for each $x,y \in X$, the function $d(x,y)$ goes to $0$ if $\eqref{t_metric}$ holds.
Let $(x^n)$ be a convergence sequence to $y$.
There is a constant $M >0$ such that for all $j \in \Z$ 
\begin{align*}
d(x^{n},y) \le  c(j,n)+ \frac{M}{2^{|j|}}
\end{align*}
where
$c(j,n) =  \sum_{i \le |j|} {|x_i^{n} - y_i|}/{2^{|i|}}$.
Notice that for each $j \in \Z$, $c(j,n) \to 0$ as $n \to \infty$.
Thus, for any $\epsilon > 0$, we can take $i_0$ and $n_0$ such that
$ {M}/{2^{|i_0|}}< {\epsilon}/{2}$ and  $c(i_0,n_0)  < {\epsilon}/{2}$,
thus completing the proof.
\end{proof}
Clearly, $\xkr$ is a closed subset of $X$, so we find the following.
\begin{co}
The set $\xkr$ is sequentially compact.
\end{co}

Now we define a renormalized action $J \colon \R^\Z \to \R$ on $\xkr$ by:
\begin{align}
\label{actionJ}
J(x) :=J_{k,\rho}(x) = \sum_{i \in \Z} A_i(x),
\end{align}
where
\begin{align}
\label{action}
A_i(x)=
\begin{cases}
      \{\sum_{j \in I_i}  h(x_j,x_{j+1})\} - |I_i|c & { i \equiv 0,2} \\
      \{ \sum_{j \in I_i} h(x_j,x_{j+1}) \} -c_i^{+}& {i \equiv 1}\\
      \{ \sum_{j \in I_i} h(x_j,x_{j+1} ) \} - c_i^{-}&  {i \equiv -1}
\end{cases},
\end{align}
$I_i = \{k_i,k_i+1,\dots,k_{i+1}-1\}$ and  $|I_i|=k_{i+1}-k_i$.
The notations $c^+$ and $c^-$ represent
\begin{align*}
c^+_i &= \min_{x \in Y^0(k_i,\rho_i) \cap Y^1(k_{i+1}, \rho_{i+1})}   \sum_{j \in I_i} h(x_j,x_{j+1}) \ and \\
c^-_i &= \min_{x \in Y^1(k_i,\rho_i) \cap Y^0(k_{i+1}, \rho_{i+1})}   \sum_{j \in I_i} h(x_j,x_{j+1}).
\end{align*}
Clearly, $A_i(x) \ge 0$ for $i \equiv \pm 1$.
Notice that $(h_1)$ implies that the values of $c^{\pm}_i$ depend on $\rho_i, \rho_{i+1}$ and $k_{i+1}-k_i$.

To check the basic properties of $J$ through the following discussion,
we first show that for an infinite transition orbit, say $x$, $J(x)$ can be finite unlike $I(x)$.
\begin{lemm}
\label{lemm:upper}
If $\rho \in P$, then there exist $y \in \xkr$ and a constant $M$ such that $J(y) \le M$ for all $k \in K$.
\end{lemm}
\begin{proof}
For each $i \equiv  1$, choose some $z^{+i}=\{z^{+i}_j\}_{j \in I_i}$ satisfying $c^+_i = \sum_{j \in I_i} h(z^{+i}_j,z^{+i}_{j+1})$.
Define  $z^{-i}=\{z^{-i}_j\}_{j \in I_i}$ for each $i \equiv  -1$ in a similar way.
We construct a test sequence $y=(y_i)_{i \in \Z}$ as follows:
\begin{align}
\begin{split}
\label{test}
y_j=
\begin{cases}
      u^0& \text{if $j \in I_i \backslash \{k_i\}$ and $i \equiv 0$}\\
      z^{+i}_j&\text{if $j \in I_i \cup \{k_{i+1}\}$ and $i \equiv 1$}\\
      u^i & \text{if $j \in I_i \backslash \{k_i\}$ and $i \equiv 2$}\\
        z^{-i}_j& \text{if $j \in I_i \cup \{k_{i+1}\}$ and $i \equiv -1$}
\end{cases}.
\end{split}
\end{align}
Since $A_i(y)=0$ for $i \equiv \pm 1$:
\begin{align*}
J(y)
&= \sum_{i \in \Z} A_i(x)= \sum_{i \in 2\Z} A_i(x)\le C \sum_{i \in \Z} \rho_i.
\end{align*}
This completes the proof.
\end{proof}

The above lemma implies that $J$ overcomes the problem referred to in Proposition \ref{prop:finite}.
Next, we show that $J$ is bounded below.

\begin{lemm}
\label{lemm:lower}
If $\rho \in P$, then $J(x) > -\infty$ for all $x \in \xkr$.
\end{lemm}
\begin{proof}
For $i \equiv 0, 2$, we see that  $A_i(x) = a_i(x)$.
By Lemma \ref{lemm:yu_lower} and $A_i(x) \ge 0$ for $i \equiv \pm 1$,
\begin{align*}
J(x) \ge -C \sum_{i \in \Z} \max \{\rho_i, \rho_{i+1}\} \ge -2C\sum_{i \in \Z} \rho_i > -\infty.
\end{align*}
By a similar argument, we obtain a constant $\gamma$ such that for any $n \in \N$,
\[
\sum_{|i| \le n} A_i(x) \ge \gamma,
\]
thus completing the proof.
\end{proof}

To ensure that $J$ has a minimizer in $\xkr$,
we present the following lemma.

\begin{lemm}
\label{lemm:welldefi}
The function $J$ is  well-defined on $\R \cup \{+\infty\}$, i.e., 
\[
\alpha:=\liminf_{n \to \infty}  \sum_{|i| \le n} A_i(x) = \limsup_{n \to \infty}  \sum_{|i| \le n} A_i(x)=:\beta.
\]
\end{lemm}
\begin{proof}
For the proof, we use a similar argument to Yu's proof of Lemma 6.1 and Proposition 2.9 in \cite{Yu2022}.
By contradiction, we assume $\alpha<\beta$.
First, we consider the case where $\beta = +\infty$.
For $\alpha< +\infty$, we take a constant $\tilde{\alpha}$ with $\tilde{\alpha} > \alpha+1-2\gamma $.
Then there are constants $n_0$ and $n_1$ such that $n_0 < n_1$ and:
\begin{align*}
\sum_{|i| \le n_0} A_i(x) \ge \tilde{\alpha} \ \text{and} \  \sum_{|i| \le n_1} A_i(x) \le \alpha+1.
\end{align*}
Then,
\[
2\gamma > \alpha+1-\tilde{\alpha}\ge
  \sum_{|i| \le n_1} A_i(x) - \sum_{|i| \le n_0} A_i(x) 
=\sum_{ i= -n_1}^{-n_0} A_i(x)  +  \sum_{i = n_0}^{n_1} A_i(x).
\]
Combining the first term and end terms implies:
\[
\sum_{ i= -n_1}^{-n_0} A_i(x) < \gamma \ \text{or} \ \sum_{ i= -n_1}^{-n_0} A_i(x)  <\gamma.
\]
This contradicts Lemma \ref{lemm:lower}.

Next, we assume $\beta<+\infty$.
Since $\alpha< \beta$, there are two sequences of positive integers $\{m_j \to \infty\}_{j \in \N}$ and $\{l_j \to \infty\}_{j \in \N}$
satisfying $m_j<m_{j+1}$, $l_j<l_{j+1}$ and $m_j+1 < l_j < m_{j+1} -1$ for all $j \in \Z_{>0}$, and:
\[
 \beta= \lim_{j \to \infty} \sum_{i \le |m_j|} A_i(x) >  \lim_{j \to \infty}  \sum_{i \le |l_j|} A_i(x) =\alpha.
\]
Then we can find $j \gg 0$ such that
\[
 \sum_{i \le |l_j|} A_i(x) - \sum_{i \le |m_j|} A_i(x) =  \sum_{i =-l_j}^{-m_j} A_i(x) + \sum_{i =m_j}^{l_j} A_i(x) < \frac{\alpha -\beta}{2}.
\]
Since $|l_j|$ and $|m_j|$ are finite for fixed $j$,
the above calculation does not depend on the order of the sums.
Thus, we obtain:
\begin{align*}
 \sum_{i =-l_j}^{-m_j} A_i(x) + \sum_{i =m_j}^{l_j} A_i(x)
&\ge   \sum_{i \in [-l_j, -m_j] \cap 2\Z} A_i(x) + \sum_{i \in [m_j, l_j] \cap 2\Z} A_i(x) \\
 &=  \sum_{i \in [-l_j, -m_j] \cap 2\Z} a_i(x) + \sum_{i \in [-l_j, -m_j] \cap 2\Z}a_i(x) 
\end{align*}
For sufficiently large $j$, we have
\[
\sum_{i \in [-l_j, -m_j] \cap 2\Z} a_i(x) \ge -C \sum_{i} |x_{-m_j}-x_{-l_j}| > \frac{\alpha-\beta}{4}
\]
and
\[
\sum_{i \in [-l_j, -m_j] \cap 2\Z} a_i(x) \ge -C \sum_{i} |x_{-m_j}-x_{-l_j}| > \frac{\alpha-\beta}{4}
\]
because  $\rho \in P$ implies
\[
\sum_{|i|>n} \rho_i \to 0 \ (n \to \infty).
\]
Therefore
\begin{align*}
 \sum_{i =-l_j}^{-m_j} A_i(x) + \sum_{i =m_j}^{l_j} A_i(x) > \frac{\alpha-\beta}{2},
 \end{align*}
which is a contradiction.
\end{proof}

%\begin{lemm}
%\label{lemm:semiconti}
%The function $J$ is  semi-continuous on $\xkr$.
%\end{lemm}
\begin{prop}
\label{prop:min}
For all $k \in K$ and $\rho \in P$, there exists a minimizer of $J$ in $\xkr$.
\end{prop}
\begin{proof}
By Lemmas \ref{lemm:upper} and \ref{lemm:lower},
we can take a minimizing sequence $x=(x^n)_{n \in \N}$ of $J$ in $\xkr$.
Since $\xkr$ is  sequentially compact, there exists $\tilde{x} \in \xkr$ with $x_{n_k} \to \tilde{x} $.
It is enough to show that
for any $\epsilon>0$, there exists $j_0$ and $n_0 \in \N$
such that:
\begin{align}
\label{eq:lower_conti}
\sum_{|i|>j_0} A_i(x^n) >-\epsilon \ (\text{for all} \ n \ge n_0)
\text{\ and\ }
\sum_{|i|>j_0} A_i(\tilde{x}) <\epsilon,
\end{align}
because if the above inequalities hold, we obtain:
\begin{align*}
J(\tilde{x})
&=\sum_{|i| \le j_0} A_i(x)+ \sum_{|i| > i_0} A_i(x)\\
&\le \lim_{n \to \infty} \sum_{|i| \le j_0} A_i(x^n) + \epsilon
= \lim_{n \to \infty} (\sum_{i \in \Z} A_i(x^n) - \sum_{|i|>j_0} A_i(x^n))+ \epsilon \\
&\le  \lim_{n \to \infty} \sum_{i \in \Z} A_i(x^n) + 2\epsilon.
\end{align*}
Using an arbitrary value of $\epsilon$, we have $J(\tilde{x}) \le  \lim_{n \to \infty} \sum_{i \in \Z} A_i(x^n) $ and
$\tilde{x}$ is the infimum (or greatest lower bound) of $J$.

We now show the first inequality of \eqref{eq:lower_conti}.
Lemma \ref{lemm:finite_bdd} implies that for any $n \in \N$ and $j \in \N$:
\begin{align}
\label{eq:bdd_low}
\sum_{|i|>j} A_i(x^n) \ge -C \sum_{|i|>j} \max \{ \rho_{2i}, \rho_{2i+1}\} \ge -C  \sum_{|i|>j}  \rho_i.
\end{align}
Note that $A_i \ge 0$ for $i \equiv 1,2$.
Since $\sum_{i \in \Z}  \rho_i$ is finite, we have $\sum_{|i|>j}  \rho_i < \epsilon/C$ for sufficiently large $j$.
Hence, the first inequality holds.

To check the second inequality, it suffices to show that $\sum_{i \in Z} J_i(\tilde{x}) $  is finite.
If $\sum_{i \in Z} J_i(y) $ is infinite, then for any $M>0$, there is a $j_0$ such that:
\begin{align*}
M \le \sum_{|i| \le j_0} A_i(\tilde{x}) = \sum_{|i| \le j_0} A_i(\lim_{n \to \infty} x^n) = \lim_{n \to \infty} \sum_{|i| \le j_0} A_i(x^n),
\end{align*}
since a finite sum $\sum_{|i| \le j_0} A_i(x)$ is continuous.
On the other hand, for any $\delta>0$, there exists $n_0$ such that if $n \ge n_0$, then
\eqref{eq:bdd_low} gives:
\begin{align*}
\sum_{|i| \le j_0} A_i(x^n) 
&=  J(x^n) - \sum_{|i| > j_0} A_i(x^n)\\
&< \inf_{x \in \xkr} J(x) + \delta + C  \sum_{|i|>j_0}  \rho_i,
\end{align*}
so $\sum_{|i| \le j_0} A_i(x^n) $ is finite, which is a contradiction.

These results and the continuity of the finite sum of $A_i$ imply that, for any $\epsilon>0$,
\begin{align*}
J(\tilde{y})
&=\sum_{|i| \le i_0} A_i(y)+ \sum_{|i| > i_0} A_i(y)
\le \lim_{n \to \infty} \sum_{|i| \le i_0} A_i(y^n) + \epsilon
\le  \lim_{n \to \infty} \sum_{i \in \Z} A_i(y^n) + 2\epsilon.
\end{align*}
By using an arbitrary value of $\epsilon$, we have $J(\tilde{y}) \le  \lim_{n \to \infty} \sum_{i \in \Z} A_i(y^n) $.
\end{proof}

%By Lemma \ref{lemm:lower} and \ref{lemm:semiconti}, we obtain

\subsection{Properties of a minimizer}
Let $x^\ast=(x^\ast_i)_{i \in Z}$ be a minimizer in Proposition \ref{prop:min}.
First, we show the properties of $A_i$ for $i \equiv 0,2$ when $\rho_i$'s are sufficiently small.

\begin{prop}
\label{prop:estimate_delta}
For any $\rho \in P$ and a positive sequence
$\delta = (\delta_i)_{i \in \Z}$,
if  $\delta$ satisfies
\[\rho_{2i} + \rho_{2i+1} +2 \max\{ \delta_{2i} ,\delta_{2i+1}\} <c_{\ast}/2C,\]
then
there exists a sequence $k=(k_i)_{i \in Z}$ that satisfies the following:
there is a sequence $l=(l_i)_{i \in \Z}$ satisfying
\begin{enumerate}
\item[$(l_1)$] $ l_{2i}, l_{2i+1} \in (k_{2i},k_{2i+1}) \cap \Z$ with $ l_{2i} < l_{2i+1} $ and
\item[$(l_2)$]  $|x^\ast_{l_i} - x^j|<\delta_i$,
where $j=0$ if $i \equiv 0,1$ and $j=1$ if $i \equiv -1,2$.
\end{enumerate}
\end{prop}
\begin{proof}
It is immediately shown from $(h_1)$, Lemmas \ref{lemm:one_delta} and \ref{lemm:finite_delta}.
\end{proof}

To see that $x^\ast$ is an infinite transition orbit,
it suffices to show that $x^\ast$ is not on the boundary of $\xkr$.
We set  $\tilde{I}_0:=(u^0,u^1) \backslash I_0$ and  $\tilde{I}_1:=(u^0,u^1) \backslash I_1$.
We assume $\eqref{gap}$ (see Section \ref{sec:pre}), i.e., $\tilde{I}_i \neq \emptyset$ for $i=0,1$.

\begin{rem}
\label{rem:rho_small}
Notice that $\rho \in \tilde{I}_i$ can be chosen as arbitrarily small.
\end{rem}

We choose $\rho$ and $k$ in the following steps:
\begin{itemize}
\item[Step $1$] 
First we take $\rho \in P$
so that
\begin{itemize}
\item[(p1)] $u^0+\rho_i \in \tilde{I}_0$ for all $i \equiv 1,2$ and $u^1-\rho_i \in \tilde{I}_1$ for all $i \equiv -1,0$ and
\item[(p2)] For any $i \in \Z$, $\rho_{2i} + \rho_{2i+1} < c_{\ast}/2C$
\end{itemize}

\item[Step $2$] To determine $|k_{2i} - k_{2i+1} |$,  
we define a positive sequence $(e^i)_{i \in \Z}$ by:
\begin{align*}
e^i=
\begin{cases}
e_0(\rho_{2i+1},\rho_{2(i+1)}), & (i : even)\\
e_1(\rho_{2i+1},\rho_{2(i+1)}) & (i : odd).
\end{cases}
\end{align*}
Choose a positive sequence $(\delta_i)_{i \in \Z}$ so that:
\begin{itemize}
\item[(d)]  $2\max \{ \delta_{2i} , \delta_{2i+1}\}< c^\ast/2C - (\rho_{2i} + \rho_{2i+1}) $
\end{itemize}
%For each $\epsilon_i$, there exists $m_i$ in Lemma \ref{lemm:finite_het}.
Then, through Proposition \ref{prop:estimate_delta} for  $(\rho_i)$ and $(\delta_i)$, we can obtain $|k_{2i}-k_{2i+1}|$ and $l_i$ for each $i \in \Z$ such that $(l_1)$ and $(l_2)$ in Proposition \ref{prop:estimate_delta} hold.
 
\item[Step $3$] 
For $\delta$ in Step $2$,  choose $(\epsilon_i)_{i \in \Z}$ so that:
\begin{itemize}
\item[(e1)]  $\epsilon_i +2C( \delta_{2i +1}  + \delta_{2(i+1)})< e^i/2 $ and
\item[(e2)] $\epsilon_i /2C < \min \{\delta_{2i+1}, \delta_{2(i+1)}\}$.
\end{itemize}
For each $\epsilon_i$, 
Lemma \ref{lemm:finite_het} gives $n_i$ and $x^i \in  \calm^i(u^0,u^1)$ such that
$\sum_{i=0}^{n-1} a_i(x) \le  c_0 + \epsilon $ for all $n \ge n_i$.
For $(n_i)_{i \in \Z}$ and $(x^i)_{i \in \Z}$, we choose  $|k_{2i+1}-k_{2(i+1)}|$ satisfying
\begin{itemize}
  \item[(k1)] $|k_{2i+1}-k_{2(i+1)}| \ge n_i$ and
  \item[(k2)] $x^i \in  \calm_0 \cap Y^0(0, \rho_{2i}) \cap Y^1(|k_{2i-1}-k_{2i}|, \rho_{2i+1})$ when $i$ is even and  $x^i \in \calm_1 \cap Y^1(0, \rho_{2i}) \cap Y^0(|k_{2i-1}-k_{2i}|, \rho_{2i+1}) $ when $i$ is odd.
\end{itemize}
\item[Step $4$] By the above steps and $k_0=0$, the sequence $k=(k_i)_{i \in \Z}$ is determined.
\end{itemize}

Now we are ready to state our main theorem when the rotation number $\alpha$ is zero.
\begin{theo}
\label{theo:inf_transition}
Suppose $\eqref{gap}$. Then,
for $\epsilon >0$, there exist two sequences $k \in K$ and $\rho \in P(\sigma)$ with $\rho_i < \epsilon$ for all $i \in Z$
such that a stationary configuration $x^\ast$ in $\xkr$ with rotation number $\alpha=0$.
\end{theo}
\begin{proof}
Fix $\epsilon >0$ arbitrarily.
Along with the above steps,
we can take $k=(k_i)$ and $\rho=(\rho_i) \in P$ so that  $\rho_i < \epsilon$ for all $i \in \Z$ by Remark \ref{rem:rho_small}.
Lemmas \ref{lemm:estimate_periodic1} and \ref{lemm:estimate_periodic2} implies that $x^\ast_i \notin \{u^0,u^1\}$ for all $i \in \Z$.
For a contradiction argument, we assume $x^\ast_{k_1} = u^0 + \rho_1$.

Hereafter, $x_i^\ast$ will be written simply as $x_i$ unless there is potential for confusion.
Let $y$ be:
\begin{align*}
y_i=
\begin{cases}
     x_i & i \in [k_1-l_1,k_2+l_2 ] \cap \Z \\
     u^1 & i > k_2+l_2\\
      u^0& i < k_1-l_1
\end{cases}.
\end{align*}
Since $y \in X^0$, Lemma \ref{lemm:gap_estimate} and $(h_1)$ yield:
\begin{align*}
c_0+e^0
&\le I(y)\\
& = \sum_{ i \in [k_1-l_1,k_2+l_2-1 ] \cap \Z} a_i(x) + \sum_ {i < k_1-l_1} a_i(x) +  \sum_{i > k_2+l_2} a_i(x ) \\
&\le \sum_{ i \in [k_1-l_1,k_2+l_2-1 ] \cap \Z} a_i(x)  + C(\delta_1 + \delta_2)
\end{align*}
By the above remark and $a_i=A_i$ for $i \equiv 0,2$,
\begin{align*}
&\sum_{i=k_1-l_1}^{k_1-1} (h(x_i,x_{i+1}) -c) + A_1(x) + \sum_{i =k_2}^{k_2+l_2-1} (h(x_i,x_{i+1})-c)\\
&\quad> c_0 + e^0 -C({\delta}_1 + \delta_2) + (|I_1|c-c^+_{1}).
\end{align*}
On the other hand,  $(h_1)$ implies that $x^0 \in \calm^0$ in Step $3$ satisfies:
\begin{align*}
\sum_{i=k_1-l_1}^{k_1-1} (h(x^0_i,x^0_{i+1}) -c) + A_1(x^0) + \sum_{i =k_2}^{k_2+l_2-1} (h(x^0_i,x^0_{i+1})-c)\
<c_0 +\epsilon_0 + (|I_1|c-c^+_{1})
\end{align*}
We define a test sequence $z=(z_i)_{i \in \Z}$ by:
\begin{align*}
z_i=
\begin{cases}
     x_{i-k_1}^0 & i \in [k_1-l_1,k_2+l_2 ] \cap \Z\\
     x_i & \text{otherwise}
\end{cases}.
\end{align*}
By Lipschitz continuity, the monotonicity of $x^0$ and (e2), we find:
\begin{align*}
|h(x_{k_1+l_1-1},x^0_{k_1-l_1}) - h(x_{k_1+l_1-1},x_{k_1+l_1}) | \le C|x^0_{k_1-l_1} - x_{k_1-l_1}| &< C\delta_1 \ and\\
|h(x^0_{k_2+l_2},x_{k_2+l_2+1}) - h(x_{k_2+l_2},x_{k_2+l_2+1}) | \le C|x^0_{k_2+l_2} - x_{k_2+l_2}| &< C\delta_2.
\end{align*}
Hence
\begin{align*}
J(z) -J(x^\ast)
&<c_0 +\epsilon_0 + C({\delta}_1 + \delta_2) -( c_0 + e^0 -C({\delta}_1 + \delta_2) )\\
&=
2(\delta_1 + \delta_2) + \epsilon_0-   e_0  < -\frac{e^0}{2} <0,
\end{align*}
which is a contradiction.
The same proof works for the remaining cases.
For example, if  $x^\ast_{k_i} = u^0 + \rho_i$ or $x^\ast_{k_i} = u^1 - \rho_i$ for all $i \in \Z$, 
a similar argument yields:
\begin{align*}
J(\hat{z}) -J(x^\ast)
&< \sum_{i \in \Z} 2(\delta_{2i+1} + \delta_{2(i+1)}) + \epsilon_i -   e_i  < - \sum_{i \in \Z} \frac{e^i}{2} <0,
\end{align*}
where $ \hat{z} = (\hat{z}_i)_{i \in \Z} $ is given by:
\begin{align*}
\hat{z}_i=
\begin{cases}
     x_{i-k_{2j+1}}^j & i \in [k_{2j+1}-l_{2j+1},k_{2(j+1)}+l_{2(j+1)} ] \cap \Z \text{\ and \ } j:even\\
     x_{i-k_{2j+1}}^j & i \in [k_{2j+1}-l_{2j+1},k_{2(j+1)}+l_{2(j+1)} ] \cap \Z \text{\ and \ } j:odd\\
     x_i & \text{otherwise}
\end{cases}.
\end{align*}
This completes the proof.
\end{proof}
%%%%%%
%以下，査読業者のチェックなし
Moreover, we immediately obtain the following corollary.
\begin{co}
For any $\alpha \in \Q$, if $\eqref{hetero_gap}$ is satisfied,
then there exists a stationary configuration $x^\ast$ of $J$ with infinite transitions and rotation number $\alpha$.
\end{co}
\begin{proof}
This follows from Proposition \ref{prop:alpha_configuration} and Theorem \ref{theo:inf_transition}.
\end{proof}

\section{The Frenkel-Kontorova model}
\label{sec:example_h}
In this section, we introduce the Frenkel-Kontorova model.
In the previous section,  we cannot generally show:
\begin{align}
\label{per_min_ineq}
h(x,y) -c \ge 0.
\end{align}
Therefore the proof of Proposition \ref{prop:finite} is somewhat technical.
However, \eqref{per_min_ineq} holds if $h$ satisfies:
\begin{align}
\label{rev}
h(x,y) = h(y,x).
\end{align}
Using $\eqref{rev}$, we can easily show the following lemma, which implies $h(x,y) -c \ge 0$.
\begin{lemm}
If a continuous function $h \colon \R^2 \to \R$ satisfies ($h_1$), ($h_3$) and $\eqref{rev}$,
then all minimizers are $(1,0)$-periodic.
\end{lemm}
\begin{proof}
From $(h_1)$, we can choice $x^\ast$ satisfying $h(x^\ast,x^\ast)=\min_{x \in \R} h(x,x)$.
By contradiction, there is $(x,y) \in \R^2$ such that $x \neq y$ and $h(x,y) < h(x^\ast,x^\ast)$. Then
\[
h(x,y) + h(y,x) < h(x^\ast,x^\ast) + h(x^\ast,x^\ast) \le  h(x,x) + h(y,y),
\]
but it contradicts $(h_3)$.
\end{proof}
We will give an example satisfying \eqref{rev}.
Consider
\begin{align}
\label{fkmodel}
h(x,y)=\frac{1}{2} \left\{ C(x-y)^2 + V(x) + V(y) \right\},
\end{align}
where $C$ is a positive constant and $V(x)=V(x+1)$ for all $x \in \R $.
It is called the Frenkel-Kontorova model \cite{AubrDaer1983, FrenKont1939}.
Since  $\partial h_1h_2 \le C < 0$,
Remark \ref{rem:delta_bangert} implies that \eqref{fkmodel} satisfies $(h_1)$-$(h_5)$.

\begin{rem}
%h6は滑らかな系では必要ない
If $h$ is of class $C^1$, we do not require $(h_6)$.
\end{rem}
If $h$ satisfies \eqref{rev},
the obtained orbits in Theorem \ref{theo:inf_transition} are monotone in the following sense:
\begin{prop}
For any $x^\ast$ in Theorem \ref{theo:inf_transition},
$(x_i^\ast)_{i={k_{2j+1}}}^{2(j+1)}$ is strictly  monotone increasing when $j$ is odd, and
it strictly monotone decreasing when $j$ is even.
\end{prop}
\begin{proof}
%monotone
We only consider the case where  $j$ is odd.
By a contradiction argument,
suppose that $x_m > x_{m+1}$ for some $m$.
By the definition of $\xkr$,
 $x_{(m-s)} < x_{m+s+1}$ for some $s$.
Now we assume $s=m$ for simplicity.
 Set $y^1=(y_i^1)_{i=0}^{2m+1}$ and $y^2=(y_i^2)_{i=0}^{2m+1}$ by:
 \begin{align*}
 y_i^1
=
\begin{cases}
      z_i^1& (i=0, \cdots, m) \\
      x_{i}& (i=m+1, \cdots, 2m+1)
\end{cases}
\text{\ , and \ }
 y_i^2
 =
\begin{cases}
      x_{i}& (i=0, \cdots, m)\\
      z_i^2 & (i=m+1, \cdots, 2m+1)
\end{cases},
 \end{align*}
 where $z^1=(z^1)_{i=0}^{m}$ is a finite minimal configuration in $\xkr$ with $z^1_0=x_0$ and $z^1_m=x_{m+1}$
 and $z^2=(z^2_i)_{i=m+1}^{2m+1}$ with $z^2_{m+1}=x_m$ and $z^2_{2m+1}=x_{2m+1}$ is given similarly.
Applying $\eqref{rev}$ and $(h_3)$, we have:
\begin{align*}
2h(x_{m},x_{m+1}) =h(x_{m},x_{m+1}) + h(x_{m+1},x_{m}) > h(x_{m},x_{m}) + h(x_{m+1},x_{m+1}),
\end{align*}
and:
\begin{align*}
&2 \sum_{i=0}^{2m} h(x_i,x_{i+1}) - \sum_{i=0}^{2m} h(y^1_{i},y^1_{i+1}) -  \sum_{i=0}^{2m} h(y^2_{i},y^2_{i+1})\\
&> \sum_{i=0}^{m-1} h(x_i,x_{i+1}) +   \sum_{i=m+1}^{2m} h(x_i,x_{i+1}) - \sum_{i=0}^{m-1} h(y^1_{i},y^1_{i+1}) -  \sum_{i=m+1}^{2m} h(y^2_{i},y^2_{i+1})\\
&= \sum_{i=0}^{m-1} h(x_i,x_{i+1}) +   \sum_{i=0}^{m-1} h(x_{2m+1-i},x_{2m+1-{(i+1)}})\\
&\quad\quad\qquad\qquad\qquad- \sum_{i=0}^{m-1} h(y^1_{i},y^1_{i+1}) -  \sum_{i=0}^{m-1} h(y^2_{2m+1-i},y^2_{2m+1-{(i+1)}})\\
&>0.
\end{align*}
Hence, it holds that  $\sum_{i=0}^{2m} h(x_i,x_{i+1}) >  \sum_{i=0}^{2m} h(y^j_{i},y^j_{i+1})$ for $i=1$ or $2$.
The same reasoning applies to the remaining cases.
This completes the proof.
\end{proof}

%Billiard maps
\if0
\subsection{Billiard maps}
Next, we consider a billiard map that does not satisfy the twist condition.
Let $f \colon \R \to \R$ be a  positive smooth function satisfying $f(x)=f(x+1)$ for all $x \in \R$.
Set an area $D=D(f)$ by
\[
D = \{ (x,y) \in \R^2 \mid -f(x) \le y \le f(x)\}.
\]
Let $s$ be an arc-length parameter, i.e., $s$ is given by
\[
s=\int_{0}^{x} \sqrt{1 + f'(\tau)^2} d \tau =:g(x)
\]
and $x$ is represented by $x=g^{-1}(s)$. Moreover, we set
\[
\tilde{f}(s) = f(g^{-1}(s)).
\]
For the above settings, we can see a variational structure of billiard maps.
Consider
\[
h(s,s') = d_{\mathrm{E}}(a_{+}(s),a_{-}(s')),
\]
where 
$a_{\pm}(s):=(g^-(s), \pm \tilde{f}(s))$
and
$d_{\mathrm{E}}$ is Euclidean metric on $\R^2$.
We check that $h$ satisfies $(h_1)$-$(h_5)$.
\fi

\section{Further discussion}
\label{sec:general}
We first see that Theorem \ref{theo:inf_transition}  implies the existence of uncountable many infinite transition orbits.
We take $k \in K$ and $\rho \in P$ in Theorem \ref{theo:inf_transition} so that for all $i \in \Z$, $k_{i+1}-k_{i} < k_{i+2}-k_{i+1} $.
For $j \in K$, set:
\begin{align*}
X_j=\left( \bigcap_{i \equiv 0,1}  Y^0(k_{j_i},\rho_i) \right) \cap \left( \bigcap_{i \equiv -1,2}  Y^1(k_{j_i},\rho_i) \right)
\end{align*}
Let  $x^\ast(j)$ be a minimizer of $J$ on $X_j$, i.e., 
\[
J(x^\ast(j))=\inf_{x \in X_j} J(x).
\]
Let $j^0=(j^0_i)_{i \in \Z}$ be a sequence of $j^0_i=i$.
The previous section deals with the case of $j=j^0$.
It is easily seen that 
if $l \neq m \in K$, then $x^\ast(l)$ and $x^\ast(m)$ are different and
we immediately get the following theorem.
\begin{theo}
Let $\# \chi $ be the number of infinite transition orbits in Theorem  \ref{theo:inf_transition}.
Then $\# \chi = \# \R$.
\end{theo}
\begin{proof}
For any real number $r \in \R_{>0}$, we can choose a corresponding bi-infinite sequence ${(a_i)}_{i \in \Z} \subset \Z_{\ge 0}$.
(For example,  when $r=12.34$, $a_{-1}=1, a_0=2, a_1=3, a_2=4$ and $a_i =0$ otherwise.)
The proof is straightforward by setting $a_i := j_{i+1} - j_i -1$ for $i \in \Z$.
\end{proof}

 To obtain a more general representation of the variational structure for transition orbits,
set $A = \{A_0,A_1\}$ and:
\begin{align*}
\xkr(A) = \left( \bigcap_{i \in A_0}  Y^0(k_{i},\rho_i) \right) \cap \left( \bigcap_{i \in A_1}  Y^1(k_{i},\rho_i) \right),
\end{align*}
where $A_0$ and $A_1$ are sets so that  $A_0 \cap A_1 = \emptyset$.
For $A$, we set:
\[
\cals:=\cals(A)=\{i \in \Z \mid i \in A_j  \text{\ and \ } i+1 \in A_{|j-1|} \, (j=0 \text{\, or \,}1) \},
\]
and
\[
\tilde{P} := \tilde{P}(\cals) = \left\{ \rho= (\rho_i)_{i \in \Z} \mid 0<\rho_i<(u^1-u^0)/2, \ \sum_{i \in \cals} \rho_i <\infty \right\}.
\]
Clearly, it always holds that $\sum_{i \in \cals} \rho_i <\infty$ if $\# \cals < \infty$.
When $A_0 \cup A_1 = \Z$, for $k \in K$ and $\rho \in\tilde{P}$,
let $\tilde{J} \colon \R^\Z \to \R$ be given by:
\begin{align}
\label{re:actionJ}
\tilde{J}(x) :=\tilde{J}_{k,\rho,A}(x) = \sum_{i \in \Z} B_i(x),
\end{align}
where
\begin{align}
\label{action}
B_i(x)=
\begin{cases}
      \{\sum_{j \in I_i}  h(x_j,x_{j+1})\} - |I_i|c & { i , i+1\in A_0}  \text{\ or  \ } { i , i+1\in A_1}  \\
      \{ \sum_{j \in I_i} h(x_j,x_{j+1}) \} -c_i^{+}& {i \in A_1 \text{\ and \ } i+1 \in A_0}\\
      \{ \sum_{j \in I_i} h(x_j,x_{j+1} ) \} - c_i^{-}&  {i \in A_0 \text{\ and \ } i+1 \in A_1}
\end{cases}.
\end{align}
For example, Section \ref{sec:pf} dealt with the case of $A_0 = \{i \in \Z \mid i \equiv 0,1 \ (\text{mod} \ 4)\}$
and $A_1=\{i \in \Z \mid i \equiv -1,2 \ (\text{mod} \ 4)\}$, so $\cals=\Z$.

 Let $x^\ast(k,\rho,A)$ be a minimizer on $\xkr(A) $, i.e.,
 $x^\ast(k,\rho,A)$ satisfies:
\begin{align*}
\tilj(x^\ast(k,\rho,A))= \inf_{x \in \xkr(A) } \tilj(x).
\end{align*}
It is easily seen that there exists a minimizer  $x^\ast(k,\rho,A)$
by a similar discussion of Proposition \ref{prop:finite} and \ref{prop:min}.
In a similar way of proving in Theorem \ref{theo:inf_transition},
we obtain the following.
\begin{theo}
Assume that $\eqref{gap}$ and there is no $i \in \Z$ such that 
$i-1 \in A_{|j-1|}$, $i \in A_{j}$, and $i+1 \in A_{|j-1|}$ for $j=0$ or $1$.
Then,
for any $\epsilon>0$,
there exist two sequences $k \in K$ and $\rho \in \tilde{P}$ with $\rho_i < \epsilon$ for all $i \in \Z$
such that $x^\ast(k,\rho,A)$ is a stationary configuration
and it has $\# \cals$ transitions.
\end{theo}
\begin{proof}
First, we consider $\# \cals=1$.
Only this case does not require $\eqref{gap}$ (but, of course, a neighboring pair $(u^0,u^1)$ is needed).
Set $A_0 =\Z_{\le 0}$, $A_1=\Z_{>0}$,
 $k \in K$ is given by $k_i=i$ \, $(i \in A_0) $  and $k_{i+1} = k_{i}+1$  \, $(i \in A_1)$, and
 $\rho \in \tilde{P}$ is a constant sequence, i.e., $\rho_i = \epsilon_0 < \epsilon$ for all $i \in \Z$.
We define $k_1$ later.
Clearly, $I(x)-\tilj(x)$ is a finite value when  $\# \cals < \infty$, and
 Proposition \ref{prop:finite} implies that $|x^\ast_i - u^0| \to 0$ as $i \to -\infty$ and   $|x^\ast_i - u^1| \to 0$ as $i \to \infty$
 for all $k, \rho$.
 By using $(h_1)$ we can assume that $u^0<x^\ast_0< u^0 + \rho_0$.
 
 Hereafter, we abbreviate $x^\ast$ as $x$.
 We first show that $x$ is strictly monotone for $i \in A_0$, i.e., $x_{i-1} < x_{i}$ for all $i \in \Z$.
 The following proof of monotonicity is similar to Proposition 3.5 of \cite{Yu2022}.
 Assume $x_j=x_{j+1}$ for some $j \in \Z$. Lemma \ref{lemm:estimate_periodic1} implies $x_j \in (u^0,u^1)$. Then $h(x_j,x_j)-c>0$.
Set $\bar{x} = (\cdots, x_{j-1}, x_{j+1}, \cdots)$ satisfying $x_0 = \bar{x}_0$ and $\bar{x} \in \xkr(A)$, then $\tilj(x)>\tilj(\bar{x})$.
Next, we assume that $(x_j -x_{j-1})(x_j -x_{j+1})>0$ for some $j \in \Z$. By using $(h_3)$,
\begin{align*}
&\tilj(x) - \tilj(\bar{x})\\
&> \tilj(x) - (\tilj(\bar{x})  + h(x_j,x_j)-c) \\
&> h(x_{j-1},x_j) + h(x_{j},x_{j+1}) -(h(x_{j+1},x_{j+1}) +h(x_{j},x_j) ) >0,
\end{align*}
which is a contradiction.

Now we check that there is $k_1 \in \Z_{>0}$ satisfying $x_{k_1}> u^1 - \rho_1$. 
We assume that for any $k_1 \in \Z$, $x_{k_1}= u^1 - \rho_1$.
The monotonicity of $x$ implies that
if $x_{k_1}= u^1 - \rho_1$, then
\[
\min_{j \in \{ 0,1\}} |u^j-x_i| \ge \min\{x^0-u^0,\rho_1\} (=:\delta)>0
\] for $i=0, \cdots, k_1$.
Applying Lemma \ref{lemm:yu_lower} and Lipschitz continuity, we obtain:
\begin{align*}
\sum_{i=0}^{k_1-1} a_i(x) \ge k_1 \phi(\delta) - C|x_{k_1}-x_0|.
\end{align*}
The left side goes to infinity as $k_1 \to \infty$, which is a contradiction.
Hence, there is $k_i \in \Z_{>0}$ such that  $x_{k_1}> u^1 - \rho_1$.
Moreover, by Lemma \ref{lemm:estimate_periodic1} and \ref{lemm:estimate_periodic2}, it holds that $x_i \in (u^0,u^1)$ for all $i \in \Z$.

Secondly, we consider the case of $\# \cals \ge 2$.
For example, set $A_0 = \Z \backslash A_1$ and $A_1 = \{1,2,\cdots, n\}$.
Lemma \ref{lemm:finite_delta} means that
for any $\rho \in \tilde{P}$, there exists $k \in K$ such that a minimizer of $\tilj_{k,\rho,\{A_0,A_1\}}$ on $\xkr(\{  A_0,A_1\})$ and $\xkr(\{ A_0, \{1,n\}\})$ are the same.
The case of $A_0 = \Z \backslash A_1$ and  $A_1 = \{1, 2\}$ is shown in a similar way to the proof of Theorem \ref{theo:inf_transition}.
In the same manner, we can show the remaining cases.
\end{proof}

\bibliographystyle{amsplain}
\bibliography{kajihara_reference}

\providecommand{\bysame}{\leavevmode\hbox to3em{\hrulefill}\thinspace}
\providecommand{\MR}{\relax\ifhmode\unskip\space\fi MR }
\providecommand{\MRhref}[2]{%
  \href{http://www.ams.org/mathscinet-getitem?mr=#1}{#2}
}
\providecommand{\href}[2]{#2}
\begin{thebibliography}{1}

\bibitem{AubrDaer1983}
S.~Aubry and P.~Y. Le~Daeron, \emph{The discrete {F}renkel-{K}ontorova model
  and its extensions. {I}. {E}xact results for the ground-states}, Phys. D
  \textbf{8} (1983), no.~3, 381--422.

\bibitem{Bangert1988}
V.~Bangert, \emph{Mather sets for twist maps and geodesics on tori}, Dynam.
  Report. Ser. Dynam. Systems Appl. \textbf{1} (1988), 1--56.

\bibitem{FrenKont1939}
J.~Frenkel and T.~Kontorova, \emph{On the theory of plastic deformation and
  twinning}, Acad. Sci. U.S.S.R. J. Phys. \textbf{1} (1939), 137--149.

\bibitem{Mather1982}
John~N. Mather, \emph{Existence of quasiperiodic orbits for twist
  homeomorphisms of the annulus}, Topology \textbf{21} (1982), no.~4, 457--467.

\bibitem{Mather1987}
\bysame, \emph{Modulus of continuity for {P}eierls's barrier}, Periodic
  solutions of {H}amiltonian systems and related topics ({I}l {C}iocco, 1986),
  NATO Adv. Sci. Inst. Ser. C: Math. Phys. Sci., vol. 209, Reidel, Dordrecht,
  1987, pp.~177--202.

\bibitem{Rabinowitz2008}
Paul~H. Rabinowitz, \emph{The calculus of variations and the forced pendulum},
  Hamiltonian dynamical systems and applications (2008), 367--390.

\bibitem{Yu2022}
G~Yu, \emph{Chaotic dynamics of monotone twist maps}, Acta Math. Sin. (Engl.
  Ser.) \textbf{38} (2022), 179--204.

\end{thebibliography}
%% \bibliography{...}

\end{document}